\newfont{\Bbb}{msbm10 scaled\magstephalf}
\documentclass[11 pt,reqno]{amsart}
\usepackage[pagewise]{lineno}
\usepackage{amsfonts}
\usepackage{amsmath, amsthm, amscd, amsfonts}
\usepackage{amssymb}
\usepackage{amsmath}
\usepackage{xcolor}
\usepackage{amscd}
\usepackage{graphicx}
 \usepackage{epstopdf}
\allowdisplaybreaks[4]
 \newtheorem{thm}{Theorem}[section]
 \newtheorem{cor}[thm]{Corollary}
 
 \newtheorem{prop}[thm]{Proposition}
 \theoremstyle{definition}
 
\theoremstyle{remark}
 \newtheorem{rem}[thm]{Remark}
 \newtheorem{exm}[thm]{Example}
 \numberwithin{equation}{section}
\newcommand{\pf}{\begin{proof}}
\newcommand{\zb}{\end{proof}}

\newcommand{\ma}{\mathcal}

\begin{document}
\title[Invariance and near invariance]{Invariance and near invariance for non-cyclic shift semigroups}
\author[Y. Liang]{Yuxia Liang}
\address{Yuxia Liang \newline School of Mathematical Sciences,
Tianjin Normal University, Tianjin 300387, P.R. China.} \email{liangyx1986@126.com}
\author[J. R. Partington]{Jonathan R. Partington}
\address{Jonathan R. Partington \newline School of Mathematics,
  University of Leeds, Leeds LS2 9JT, United Kingdom.}
 \email{J.R.Partington@leeds.ac.uk}
  \begin{abstract}
We characterize the subspaces of $H^2(\mathbb D)$ that are 
invariant under the semigroups generated by two higher-order shifts $S^m$ and $S^{n}$.
Complete descriptions are obtained for both invariant and nearly invariant
subspaces associated with these operators and their adjoints. The approach
is based on vector-valued Hardy spaces, the Beurling--Lax theorem, and
matrix-valued inner functions. Finally, we apply these results to
non-cyclic shift semigroups and to Toeplitz operators induced by finite
Blaschke products.
\end{abstract}

\maketitle

\section{Introduction}
Invariant subspaces of operators on Hilbert spaces have been a central
topic connecting operator theory, complex analysis, and functional
analysis. A fundamental result is Beurling's theorem, which characterizes
the invariant subspaces of the unilateral shift $S$ on $H^2(\mathbb D)$
as $\theta H^2(\mathbb D)$, where $\theta$ is an inner function
(see, e.g., \cite[A.~Cor.~1.4.1]{Nik}).

For an isometry $T$ on $H^2(\mathbb{D})$, a closed subspace $\mathcal M$ is called nearly
$T^*$-invariant if
\[
f\in T(H^2)\cap\mathcal M\quad\Longrightarrow\quad T^*f\in\mathcal M .
\]
In particular, nearly $S^*$-invariant subspaces have been extensively
studied by Hayashi, Hitt and Sarason in connection with Toeplitz kernels and model
spaces; see \cite{Ha,hitt,Sa1,Sa2} and related developments
\cite{ABBH,ALS,CaP1,CaP,CaP2,CCP10,GMR,HR,LP5,ryan}.  More recently, the notion of nearly invariant subspaces has been extended
to left invertible operators and $C_0$ semigroups, leading to a broader
framework for studying invariant structures of operator semigroups
\cite{LP2,LP3,LP4}. 

Motivated by these developments, we study a different type of invariance:
simultaneous invariance under distinct powers of the unilateral shift.
Although Beurling's theorem gives a complete description for $S$,
the structure of subspaces invariant under two higher powers
$S^m$ and $S^n$ with $n>m$ remains largely unexplored. In particular,
we ask whether there exist nontrivial subspaces invariant under both
operators but not invariant under $S$.

Simple examples show that such phenomena do indeed occur. For instance,
the closed span generated by the set
$\{z^{2n+3m}:m,n\in\mathbb N_0\}$ is invariant under both $S^2$ and
$S^3$, but is neither $S$-invariant nor nearly $S^*$-invariant.
Similar examples arise for other pairs of powers. These observations
indicate that simultaneous higher-order invariance possesses a richer
structure than the classical shift-invariant theory.

The main purpose of this paper is to provide a complete characterization
of these subspaces. Our approach is based on the identification of
$H^2(\mathbb D)$ with vector-valued Hardy spaces through the isometric
isomorphism $T_m$ in \eqref{Tfgm}, together with the Beurling--Lax theorem and
matrix-valued inner functions.

A key operator in this analysis is the matrix-valued isometry
\begin{eqnarray} \Sigma_{m,\gamma}^k:=\left(\begin{array}{cc}
O_{\gamma\times (m-\gamma)} & S^{k+1}I_{\gamma\times \gamma} \\
S^kI_{(m-\gamma)\times (m-\gamma)}& O_{(m-\gamma)\times \gamma}
\end{array}\right)_{m\times m},\label{Sigmam0}\end{eqnarray} where $m \geq 2, \ k \geq 1 \ \text{and} \ \gamma \in \{1, 2, \cdots, m-1\}.$
Here $O_{m\times n}$ and $I_{m\times n}$ denote the
zero and identity matrices, respectively. Besides, let $H^{\infty}_{r \times r}$ be the collection of all bounded analytic $r\times r$ matrix-valued functions for $r\geq1$.

The paper is organized as follows. Section 2 characterizes the subspaces invariant under $S^2$ and $S^{2k+1}$ and provides examples illustrating the difference from classical $S$-invariance. Section 3 develops the
corresponding theory of nearly invariant subspaces for the adjoint
operators. Section 4 extends the results to general pairs
$S^m$ and $S^{km+\gamma}$ and applies them to non-cyclic shift
semigroups. Finally, we transfer these characterizations to Toeplitz
operators induced by finite Blaschke products.

\section{Invariance under both $ S^2$ and $S^{2k+1}$}
In this section, we characterize the subspaces invariant under both
$S^{2}$ and $S^{2k+1}$ for $k\geq1$ by using the matrix
$\Sigma_{2,1}^{k}$ (corresponding to $m=2$ and $\gamma=1$ in
\eqref{Sigmam0}) and the Beurling--Lax theorem on vector-valued Hardy
spaces.

We briefly recall that the vector-valued Hardy space
$H^2(\mathbb D,\mathbb C^n)$ consists of analytic functions
$F=(f_1,\cdots,f_n)$ with $f_j\in H^2(\mathbb D)$.
An analytic function
$\Phi:\mathbb D\to\mathcal L(\mathbb C^r,\mathbb C^n)$ is called inner
if $\Phi\in H^\infty(\mathbb D,\mathcal L(\mathbb C^r,\mathbb C^n))$
and $\Phi(e^{it})$ is an isometry almost everywhere on $\mathbb T$.
We shall use the following form of the Beurling--Lax theorem
(see, e.g., \cite{Lax,Par}).

\begin{thm}\label{Beurling-Lax} A subspace
  $\ma{K}$ of $ H^2(\mathbb{D}, \mathbb{C}^n)$ is  $S$-invariant   if and only if  $\ma{K}=\Theta_{n\times r} H^2(\mathbb{D}, \mathbb{C}^r)$, where $0\leq r\leq n$ and $\Theta_{n\times r}$ is an operator-valued inner function with values in $\ma{L}(\mathbb{C}^r, \mathbb{C}^n)$. 
\end{thm}

Meanwhile, the corresponding model space $K_{\Theta_{n\times r}}:=H^2(\mathbb{D}, \mathbb{C}^n)\ominus \Theta_{n\times r} H^2(\mathbb{D}, \mathbb{C}^r)$ is $S^*$-invariant for $0\leq r\leq n$.

 Although the $S^2$-invariant subspaces have been studied by several
approaches (see, e.g., \cite{GGP22,LS97}), we adopt a vector-valued
Hardy space framework. Indeed, every $f\in H^2(\mathbb D)$ admits the
decomposition $f(z)=f_0(z^2)+zf_1(z^2)$ with $f_0,f_1\in H^2(\mathbb D)$. Thus the mapping \begin{eqnarray}T_2 (f_0(z),f_1(z))=f_0(z^{2})+z f_1(z^{2})\label{Tfg}\end{eqnarray}  defines an isometric isomorphism from
$H^2(\mathbb D,\mathbb C^2)$ onto $H^2(\mathbb D)$. Since $$ S^2( f_0(z^2)+zf_1(z^2))=(Sf_0)(z^2)+z(Sf_1)(z^2),$$    the following commutative diagram  holds:
  \begin{eqnarray}\begin{CD}
H^2(\mathbb{D},\mathbb{C}^2)@> S>>H^2(\mathbb{D},\mathbb{C}^2)\\
@VV T_2     V @VV T_2 V\\
H^2(\mathbb{D}) @>S^2>> H^2(\mathbb{D}).
\end{CD}\label{commute10}\end{eqnarray}

Hence, $S^2$ on $H^2(\mathbb D)$ is equivalent, via $T_2$, to the  vector-valued shift
$S$ on $H^2(\mathbb D,\mathbb C^2)$. Theorem \ref{Beurling-Lax} then yields
the following description of $S$-invariant subspaces. 

\begin{prop}\label{inv-subspace} A subspace $\ma{N}$ of  $H^2(\mathbb{D}, \mathbb{C}^2)$ is $S$-invariant  if and only if  $\ma{N}=\Theta_{2\times r} H^2(\mathbb{D},\mathbb{C}^r)$, where  $0\leq r\leq 2$ and  $\Theta_{2\times r}$ is an operator-valued inner function with values in $\ma{L}(\mathbb{C}^r, \mathbb{C}^2)$. \end{prop}

By Proposition \ref{inv-subspace}, we can proceed to the $S^2$ and $S^{2k+1}$-invariant subspace in $H^2(\mathbb{D})$.

\begin{thm}\label{thm Sk2}   For $k\geq 1,$ a subspace $\ma{M}$ of $H^2(\mathbb{D})$  is invariant under both  $S^2$ and $S^{2k+1}$  if and only if $\ma{M}=T_2(\Theta_{2\times r} H^2(\mathbb{D},\mathbb{C}^r))$, where $0\leq r\leq 2$, the inner function $\Theta_{2\times r}$ takes values in $\ma{L}(\mathbb{C}^r,
\mathbb{C}^2)$ and satisfies $\Theta_{2\times r}^{*} \Sigma_{2,1}^k \Theta_{2\times r} \in H^{\infty}_{r \times r}$, and  $T_2$ is an isometric isomorphism as  in \eqref{Tfg}. \end{thm}

\begin{proof}
We first prove the sufficiency. By Proposition \ref{inv-subspace},
$\Theta_{2\times r}H^2(\mathbb D,\mathbb C^r)$ is invariant under the vector-valued shift. Hence, by \eqref{commute10},
$ \mathcal M=T_2(\Theta_{2\times r}H^2(\mathbb D,\mathbb C^r)) $
is $S^2$-invariant.

For $F\in H^2(\mathbb D,\mathbb C^r)$, we have
\[
T_2(\Theta_{2\times r}F)
=
(1,z)(\Theta_{2\times r}F)(z^2).
\]
Therefore,
\[
\begin{aligned}
S^{2k+1}T_2(\Theta_{2\times r}F)
&=
(z^{2k+1},z^{2k+2})
(\Theta_{2\times r}F)(z^2)
\\
&=
(1,z)
\left(\Sigma_{2,1}^k\Theta_{2\times r}F\right)(z^2).
\end{aligned}
\]
Since
$\Theta_{2\times r}^{*}\Sigma_{2,1}^k\Theta_{2\times r}\in
H^\infty_{r\times r}$,
the multiplier theorem for vector-valued Hardy spaces
(\cite[Theorem 4.14]{AM02}) gives
\[
\Sigma_{2,1}^k\Theta_{2\times r}H^2(\mathbb D,\mathbb C^r)
\subset
\Theta_{2\times r}H^2(\mathbb D,\mathbb C^r).
\]
Therefore
$S^{2k+1}\mathcal M\subset\mathcal M$, proving the sufficiency.

Conversely, let $\mathcal M$ be invariant under both $S^2$ and $S^{2k+1}$.
By Proposition \ref{inv-subspace},
\[
\mathcal M
=
T_2(\Theta_{2\times r}H^2(\mathbb D,\mathbb C^r))
\]
for some inner matrix function $\Theta_{2\times r}$.

The $S^{2k+1}$-invariance is equivalent to
\[
\Sigma_{2,1}^k\Theta_{2\times r}
H^2(\mathbb D,\mathbb C^r)
\subset
\Theta_{2\times r}
H^2(\mathbb D,\mathbb C^r).
\]
Hence
$\Theta_{2\times r}^{*}\Sigma_{2,1}^k\Theta_{2\times r}$
acts boundedly on
$H^2(\mathbb D,\mathbb C^r)$.
Applying again
\cite[Theorem 4.14]{AM02},
we conclude that
\[
\Theta_{2\times r}^{*}\Sigma_{2,1}^k\Theta_{2\times r}
\in H^\infty_{r\times r},
\]
which completes the proof.
\end{proof}

By Theorem \ref{thm Sk2},  a function $f_0(z^2)+zf_{1}(z^2)\in \ma{M}^{\perp}$ if and only if
\begin{align*}
 \left\langle (f_{0},  f_1 )^{\mathsf t}, \Theta_{2\times r} F \right\rangle=0\;\mbox{for all}\;  F  \in H^{2}(\mathbb{D},\mathbb{C}^r).  \end{align*}

Now we conclude the  result  on  $(S^2)^*$ and $(S^{2k+1})^*$-invariant subspaces.
\begin{prop}\label{prop S*k2} For $k\geq 1,$ a subspace $\ma{N}$ of $H^2(\mathbb{D})$ is invariant under both  $(S^2)^*$ and $(S^{2k+1})^*$ if and only if  $ \ma{N}=T_2(K_{\Theta_{2\times r}}),$ where $0\leq r\leq 2$, the inner function $\Theta_{2\times r}$ takes values in $\ma{L}(\mathbb{C}^r,
\mathbb{C}^2)$ and satisfies $\Theta_{2\times r}^{*} \Sigma_{2,1}^k \Theta_{2\times r} \in H^{\infty}_{r \times r}$, and  $T_2$ is an isometric isomorphism as  in \eqref{Tfg}.\end{prop}
 Letting $k=1$ in Theorem \ref{thm Sk2} and Proposition \ref{prop S*k2}, the characterizations on both $S^2$ and $S^3$-invariant subspace, and  both $(S^2)^*$ and $(S^3)^*$-invariant subspace follow. Two nontrivial  both $S^2$  and $S^3$-invariant subspaces are presented, which are not invariant under $S.$

\begin{exm}($1$) For $r=2$ and let  $\Theta_{2\times 2}=
\left(\begin{array}{lll}
\theta & 0 \\
0 & \phi
\end{array}\right) $, where $\theta$ and $\phi$ are nonconstant inner functions in $H^2(\mathbb{D})$. Theorem \ref{thm Sk2} indicates that the $S^2$ and $S^3$-invariant subspace $\ma{M}\subset H^2(\mathbb{D})$ can be expressed as
\begin{eqnarray*}
\ma{M}=\{f(z)= f_0(z^{2})+zf_1(z^{2}),\; \Theta_{2\times 2}^{*} \Sigma_{2,1}^1 \Theta_{2\times 2} \in H^{\infty}_{2 \times 2},\;f_0\in \theta H^{2}(\mathbb{D}), f_1\in \phi H^2(\mathbb{D})\}.\label{Mfg}
 \end{eqnarray*} The condition $$\Theta_{2\times 2}^{*} \Sigma_{2,1}^1 \Theta_{2\times 2} =\left(\begin{array}{cc}
0 & \theta^*S^2\phi \\
\phi^* S\theta & 0
\end{array}\right)\in H^{\infty}_{2 \times 2}$$ yields the inner functions $\theta$ and $\phi$  satisfy the divisions $\theta|S^2 \phi$ and $\phi| S\theta.$

Hence we further assume $\theta(z)=z^{j} \psi(z)$ and $\phi(z)=z^{k} \psi(z)$, where $\psi$ is inner, $\psi(0) \neq 0$ and $j \leq k+2$ and $k \leq j+1$, so that $-1 \leq j-k \leq 2$. That is, the matrix $\Theta_{2\times 2}$ can be
{\small  \begin{eqnarray}
\psi(z)\left(\begin{array}{lll}
z^k & 0 \\
0 & z^k
\end{array}\right), \;
\psi(z)\left(\begin{array}{lll}
z^{k+1} & 0 \\
0 & z^k
\end{array}\right),\;
\psi(z)\left(\begin{array}{lll}
z^{k+2} & 0 \\
0 & z^k
\end{array}\right),\;
\psi(z)\left(\begin{array}{lll}
z^{k} & 0 \\
0 & z^{k+1}
\end{array}\right)\label{Theta}
\end{eqnarray}}with inner $\psi(0)\neq 0$ and integer $ k\geq 0.$ Then  there are four types of $S^2$ and $S^3$-invariant subspace corresponding to $\Theta_{2\times 2}$  in \eqref{Theta}. But they are  not $S$-invariant. To be precise, let $f(z)=f_0(z^2)+zf_1(z^2)\in \ma{M}$; it follows that
\begin{align*}&Sf(z)=zf_0(z^2)+(Sf_1)(z^2)\in \ma{M}\implies \binom{Sf_1}{f_0}\in \Theta_{2\times 2}H^2(\mathbb{D},\mathbb{C}^2)\\& \implies \left(\begin{array}{lll}
S & 0 \\
0 & 1
\end{array}\right)\binom{f_1}{f_0}\in \Theta_{2\times 2}H^2(\mathbb{D},\mathbb{C}^2)\implies\Theta_{2\times 2}^*\left(\begin{array}{lll}
S & 0 \\
0 & 1
\end{array}\right)=\left(\begin{array}{lll}
z\theta^* & 0 \\
0 & \phi^*
\end{array}\right)\in H_{2\times 2}^\infty,\;\;\;\; \end{align*}
which is an contradiction.

Meanwhile, Proposition \ref{prop S*k2} illustrates there are four types of $(S^2)^*$ and $(S^3)^*$-invariant subspace, which can be expressed as
 \begin{eqnarray*}
\ma{M}^{\perp}=\{f(z)=f_0(z^{2})+z f_1(z^{2}),\;(f_{0},  f_1 )  \in K_{\Theta_{2\times 2}}\;\mbox{with}\; \Theta_{2\times 2}\;\mbox{in}\; \eqref{Theta}\}.
 \end{eqnarray*}

($2$) For $r=1$ and let $\Theta_{2\times 1}(z)=\frac{1}{\sqrt{5}}\left(\begin{array}{lll}
\theta  \\
 2\theta
\end{array}\right)$ with nonconstant inner function $\theta$,  Theorem \ref{thm Sk2} and Proposition \ref{prop S*k2} imply the nontrivial $S^2$ and $S^3$-invariant subspace is
\begin{eqnarray}\ma{M} =\{(1 + 2z)f(z^2):\; f\in\theta H^2(\mathbb{D})\},\label{ex2}\end{eqnarray}
and the $(S^2)^*$ and $(S^3)^*$-invariant subspace behaves as
\begin{eqnarray*}
\ma{M}^{\perp}=\{f(z)=f_0(z^{2})+z f_1(z^{2}),\;(f_{0},  f_1 ) \in K_{\Theta_{2\times 1}}\}.
 \end{eqnarray*}

However, the subspace $\ma{M}$ in \eqref{ex2} is not invariant under $S$. This is because
\begin{eqnarray*}&&S(1+2z)f(z^2)=2(Sf)(z^2)+zf(z^2)\in \ma{M}\implies \binom{2S}{1}f\in \Theta_{2\times 1}H^2(\mathbb{D})\\&&\implies \Theta_{2\times 1}^*\binom{2S}{1}=2(z+1)\theta^* \in H^\infty, \end{eqnarray*} which is a contradiction. \end{exm}

\begin{rem}  Invariance under $S^{2}$ and $S^{3}$ implies invariance under all powers of $S$ except perhaps $S$ itself, so this would be the smallest class of this type strictly containing the Beurling-type spaces. \end{rem}

\section{ near invariance under both $(S^{2})^{*}$ and $(S^{2k+1})^*$}
Let $\ma{M}$ be a nearly $(S^{2})^{*}$ and $(S^{2k+1})^*$-invariant subspace for $k\geq 1$. Set $$X:=\ma{M}\ominus (\ma{M}\cap S^2 H^2(\mathbb{D})).$$  Let $ \hat{k}_0(z)$ and $\hat{k}_1(z)$  be the orthogonal projections of $1$  and $z$ on $X$, respectively.  We apply the Gram--Schmidt orthogonalization to the set $\{ \hat{k}_0(z),  \hat{k}_1(z)\}$ and then normalize them to obtain an orthonormal set denoted  by $\{e_0(z), e_1(z)\}$, i.e., $E_2(z)$ is the $2\times 1$ matrix whose columns are the orthonormalization of the reproducing kernel and  the first-derivative kernel for $\ma{M}$ at $0.$  Denote by $P_2$  the orthogonal projection onto   $X$ and    \begin{eqnarray}E_2(z)=\left(e_0, \;e_1
 \right)^{\mathsf t}(z).\label{E2e}\end{eqnarray}
Since $\ma{M}$ is nearly $(S^2)^*$-invariant,  every $f\in \ma{M}$ admits the decomposition
$$f(z)=(a_0,\;b_0)E_2(z)+S^2 f_1(z).$$ Considering $f_1\perp e_0$ and $f_1\perp e_1$, the near $(S^2)^*$-invariance of $\ma{M}$ implies  $f_1\in \ma{M}$ and $f_1(z) =(S^2)^* (I-P_2)f(z)$. Meanwhile,
 $\|f\|^2=|a_0|^2+|b_0|^2+\|f_1\|^2.$  Then we repeat this process for $f_1$, getting that
\begin{eqnarray}f_1(z)=(a_1,\;b_1)E_2(z)+S^2 f_2(z)\label{f1}\end{eqnarray} with $f_2\in \ma{M}$, $f_2(z)=(S^2)^* (I-P_2)f_1(z)= [(S^2)^* (I-P_2)]^2f(z)$, and
 $\|f_1\|^2=|a_1|^2+|b_1|^2+\|f_2\|^2.$
Continuing recursively, we obtain
$$f_j(z)=(a_j,\;b_j)E_2(z)+S^2 f_{j+1}(z)$$ with $f_{j+1}\in \ma{M}$, $f_{j+1}(z)=[(S^2)^* (I-P_2)]^{j+1}f(z)$,  and  $\|f_j\|^2=|a_j|^2+|b_j|^2+\|f_{j+1}\|^2.$
 Iterating these decompositions gives
$$ f(z)=\left(\sum_{l=0}^{j}a_l z^{2l},\;\sum_{l=0}^jb_lz^{2l}\right)E_2(z)
+z^{2(j+1)}f_{j+1}(z)$$ and $$\|f\|^2=\sum_{l=0}^j |a_l|^2+\sum_{l=0}^j|b_l|^2+\|f_{j+1}\|^2.$$

We claim that $\|f_{j+1}\|\rightarrow 0$ as $j\rightarrow \infty$. Since  $(S^2)^*$ is $C_{0.}$ and the range of the projection $P_2$ is two-dimensional, 
\cite[Lemma 3.3]{BT00} implies that $ (I-P_2)(S_2)^*$ is $C_{0.}.$ This together with \([(S^2)^* (I-P_2)]^{j+1}= (S^2)^*[(I-P_2)(S_2)^*]^{j} (I-P_2)\) imply that  $(S^2)^* (I-P_2)$ is $C_{0.}$, i.e., $\| [(S^2)^*(I-P_2)]^{j+1} f\|\rightarrow 0$ as $j\rightarrow \infty,$ and hence  
\begin{eqnarray}f(z)=\left(\sum_{l=0}^{\infty}a_l (z^{2})^{l},\;\sum_{l=0}^\infty b_l(z^{2})^{l}\right)E_2(z),\;z\in\mathbb{D},\label{S2}\end{eqnarray} where the sums converge in $H^2(\mathbb{D})$ norm and \begin{eqnarray*}\|f\|^2=\sum_{l=0}^\infty|a_l|^2+
\sum_{l=0}^\infty|b_l|^2.\end{eqnarray*}
Meanwhile, observing from  \eqref{f1} and \eqref{S2}, it follows that \begin{eqnarray}f_1(z)=\left(\sum_{l=0}^{\infty}a_{l+1} (z^{2})^{l},\;\sum_{l=0}^\infty b_{l+1}(z^{2})^{l}\right)E_2(z) \in \ma{M}.\label{f11}\end{eqnarray}
Then $f\in \ma{M}$ if and only if $f(z)=\left(k_0, \;k_1\right)(z^2)E_2(z)$ with \begin{eqnarray}\|f\|^2=\|k_0\|^2+\|k_1\|^2,\label{norm}\end{eqnarray}
where $(k_0,\; k_1)(z)$ lies in a subspace $K\subset H^2(\mathbb{D},\mathbb{C}^2),$ with $k_0(z)=\sum_{l=0}^\infty a_lz^l$ and
$k_1(z)=\sum_{l=0}^\infty b_lz^l$. By \eqref{norm}, we see that $K\subset H^2(\mathbb{D},\mathbb{C}^2)$ is closed. Furthermore, \eqref{f11} implies $$f_1(z)= \left(S^*k_0,\; S^*k_1\right)(z^2)E_2(z)\in \ma{M},$$ that is to say $(S^*k_0,\;S^*k_1)(z)\in K$. So $K$ is  $S^*$-invariant in $H^2(\mathbb{D},\mathbb{C}^2)$.

Conversely, assume that \begin{align}\ma{M}=\left\{\left(k_0,\; k_1\right)(z^2)E_2(z) :\; (k_0, \;k_1)(z)\in K\right\}, \label{spaceM}\end{align}   where $K$  is  $S^*$-invariant in $H^2(\mathbb{D},\mathbb{C}^2)$. Then $\ma{M}$ is nearly $(S^2)^*$-invariant.

By Theorem \ref{Beurling-Lax}, 
 $K=K_{\Theta_{2\times r}}:=H^2(\mathbb{D},\mathbb{C}^2)\ominus \Theta_{2\times r}H^2(\mathbb{D},\mathbb{C}^r),$  where  $0\leq r\leq 2$ and the inner function $\Theta_{2\times r}\in H^\infty(\mathbb{D},\ma{L}(\mathbb{C}^r,\mathbb{C}^2))$. By \eqref{spaceM}, there is a surjective isometry  $J_2:\;\mathcal{M}\rightarrow K_{\Theta_{2\times r}}$ defined by
$$J_2\left(G(z^2) E_2(z)\right)=G(z)\;\mbox{with}\;G\in H^2(\mathbb{D},\mathbb{C}^2),$$ which is unitary.  We obtain the following characterization.

\begin{prop}\label{prop nearS2} A subspace $\ma{M}$ of $H^2(\mathbb{D})$ is nearly $(S^2)^*$-invariant if and only if  there exists a unitary map
\begin{eqnarray*}J_2:\;\mathcal{M}\rightarrow K_{\Theta_{2\times r}}\;\; \mbox{defined by}\;\;G(z^2) E_2(z)\mapsto G(z),\label{J2k}\end{eqnarray*} where the inner function $\Theta_{2\times r}$ takes values in $\ma{L}(\mathbb{C}^r,\mathbb{C}^2)$  for  $0\leq r\leq 2$.
\end{prop}

Using Proposition \ref{prop nearS2}, the simultaneously near $(S^2)^*$ and $(S^{2k+1})^*$-invariant subspaces are obtained as below, which also provides all nearly invariant subspaces for the non-cyclic shift semigroup $\{(S^*)^{2n} (S^*)^{(2k+1)m}:\;m, n\in \mathbb{N}_0,\;k\geq 1\}$ generated by two elements $(S^*)^{2}$ and $(S^*)^{2k+1}$.
\begin{thm}\label{nearS2k1}For $k\geq 1,$  a subspace $\ma{M}$ of $H^2(\mathbb{D})$  is nearly $(S^2)^*$  and $(S^{2k+1})^*$-invariant if and only if  there exists a unitary map 
\begin{eqnarray*}J_2:\;\mathcal{M}\rightarrow K_{\Theta_{2\times r}}\;\; \mbox{defined by}\;G(z^2) E_2(z)\mapsto G(z), \end{eqnarray*} where the inner function $\Theta_{2\times r}$ takes values in $\ma{L}(\mathbb{C}^r,
\mathbb{C}^2)$ and satisfies $\Theta_{2\times r}^{*} \Sigma_{2,1}^k \Theta_{2\times r} \in H^{\infty}_{r \times r}$ for  $0\leq r\leq 2$.  
 \end{thm}

\begin{proof}
We first prove the sufficiency. By Proposition \ref{prop nearS2}, there exists a unitary map
\(J_2:\mathcal M\to K_{\Theta_{2\times r}}\) such that every
\(f\in\mathcal M\) can be written as
\[
f(z)=(k_0,k_1)(z^2)E_2(z),
\]
where \((k_0,k_1)\in K_{\Theta_{2\times r}}\). And $\ma{M}$ is nearly $(S^2)^*$-invariant.

Assume that \(f^{(n)}(0)=0\) for \(0\le n\le 2k\).
Then the definition of \(E_2\) in \eqref{E2e} yields
\[
(S^{2k+1})^*f(z)
=
\big((S^{k+1})^*k_0,(S^k)^*k_1\big)(z^2)
\begin{pmatrix}
S&0\\
0&S^*
\end{pmatrix}
E_2(z).
\]

Thus, under the unitary map \(J_2\), it remains to verify that
\[
\big((S^{k+1})^*k_0,(S^k)^*k_1\big)(z)
\begin{pmatrix}
S&0\\
0&S^*
\end{pmatrix}
\in K_{\Theta_{2\times r}} .
\]

By the definitions of \(E_2\) in \eqref{E2e} and $T_2$ in \eqref{Tfg}, the above condition is equivalent to
\begin{align}
(\Sigma_{2,1}^{k})^*
(k_0,k_1)^{\mathsf t}
\in K_{\Theta_{2\times r}}.\label{aimsk}
\end{align}

For any \(F\in H^2(\mathbb D,\mathbb C^r)\), it follows that
\[
\begin{aligned}
 \left\langle
(\Sigma_{2,1}^{k})^*
(k_0,k_1)^{\mathsf t},
\Theta_{2\times r}F
\right\rangle =
\left\langle
(k_0,k_1)^{\mathsf t},
\Sigma_{2,1}^{k}\Theta_{2\times r}F
\right\rangle .
\end{aligned}
\]

Since
\(\Theta_{2\times r}^{*}\Sigma_{2,1}^{k}\Theta_{2\times r}
\in H^\infty_{r\times r}\), we obtain
\[
\Sigma_{2,1}^{k}\Theta_{2\times r}F
=
\Theta_{2\times r}
(\Theta_{2\times r}^{*}\Sigma_{2,1}^{k}\Theta_{2\times r})F
\in
\Theta_{2\times r}H^2(\mathbb D,\mathbb C^r).
\]
Therefore the above inner product is zero, because
\((k_0,k_1) \in K_{\Theta_{2\times r}}\). Hence \eqref{aimsk} holds
and consequently \(\mathcal M\) is nearly
\((S^{2k+1})^*\)-invariant.

Conversely, assume that \(\mathcal M\) is nearly invariant with respect
to both \((S^2)^*\) and \((S^{2k+1})^*\).
By Proposition \ref{prop nearS2}, there exists a unitary map
\(J_2:\mathcal M\to K_{\Theta_{2\times r}}\), where
\(\Theta_{2\times r}\) is inner and \(0\le r\le2\).   The near invariance under \((S^{2k+1})^*\) implies that
\[
(\Sigma_{2,1}^{k})^*
(k_0,k_1)^{\mathsf t}
\in K_{\Theta_{2\times r}}
\]
for every \((k_0,k_1) \in K_{\Theta_{2\times r}}\).
Hence, for every \(F\in H^2(\mathbb D,\mathbb C^r)\),

\[
\left\langle
(k_0,k_1)^{\mathsf t},
\Sigma_{2,1}^{k}\Theta_{2\times r}F
\right\rangle=0 .
\]

Since \(K_{\Theta_{2\times r}}^\perp
=\Theta_{2\times r}H^2(\mathbb D,\mathbb C^r)\), it follows that
$ \Sigma_{2,1}^{k}\Theta_{2\times r}F
\in
\Theta_{2\times r}H^2(\mathbb D,\mathbb C^r).$ 
Therefore
\(\Theta_{2\times r}^{*}\Sigma_{2,1}^{k}\Theta_{2\times r}\)
is a bounded multiplier on
\(H^2(\mathbb D,\mathbb C^r)\), which gives
\[
\Theta_{2\times r}^{*}\Sigma_{2,1}^{k}\Theta_{2\times r}
\in H^\infty_{r\times r}.
\]
The proof is complete.
\end{proof}

We next represent two nontrivial  nearly $(S^2)^*$ and $(S^3)^*$-invariant subspaces by  $\Sigma_{2,1}^1=\left(\begin{array}{ll}0& S^2\\S &0\end{array}\right)$. These examples show that simultaneous near invariance with
respect to higher order shifts may occur even when the classical $S^*$-near invariance fails.
\begin{exm}
First, let
\(
\mathcal M_1=
\{a(1+z+z^2)+bz(1+2z):a,b\in\mathbb C\}.
\)
Then $\mathcal M_1$ is not nearly $S^*$-invariant. A direct computation
gives
\[
E_2(z)=
\begin{pmatrix}
(5+2z-z^2)/\sqrt{30}\\
z(1+2z)/\sqrt5
\end{pmatrix},
\]
and
\[
\mathcal M_1=\{(a, b)E_2(z):a,b\in\mathbb C\}.
\]
Hence the unitary identification
$J_2:\mathcal M_1\to K_{\Theta_{2\times1}}$ is given by
$(a,b)E_2(z)\mapsto(a,b)$, where
\[
\Theta_{2\times1}
=\frac1{\sqrt2}
\begin{pmatrix}z\\z\end{pmatrix}.
\]
Since
\[
\Theta_{2\times1}^{*}\Sigma_{2,1}^{1}\Theta_{2\times1}
=\frac{z+z^2}{2}\in H^\infty,
\]
Theorem \ref{nearS2k1} implies that $\mathcal M_1$ is nearly
$(S^2)^*$  and $(S^3)^*$-invariant.

Similarly, consider
\(
\mathcal M_2=
\{(a+bz+cz^3)(1+z):a,b,c\in\mathbb C\}.
 \)
Although $\mathcal M_2$ is not nearly $S^*$-invariant, we have
\[
E_2(z)=
\begin{pmatrix}
(2-z)(1+z)/\sqrt6\\
z(1+2z)/\sqrt5
\end{pmatrix}
\]
and
\[
\mathcal M_2=
\{(\alpha+\beta z^2,\gamma-2\beta z^2)E_2(z):
\alpha,\beta,\gamma\in\mathbb C\}.
\]
The corresponding model space is
\(
K_{\Theta_{2\times1}}
=
\{(\alpha+\beta z,\gamma-2\beta z):
\alpha,\beta,\gamma\in\mathbb C\},
\)
where
\[
\Theta_{2\times1}
=
\frac1{\sqrt2}
\begin{pmatrix}z^2\\z^2\end{pmatrix}.
\]
Again,
\[
\Theta_{2\times1}^{*}\Sigma_{2,1}^{1}\Theta_{2\times1}
=\frac{z+z^2}{2}\in H^\infty,
\]
and hence $\mathcal M_2$ is nearly
$(S^2)^*$  and $(S^3)^*$-invariant by Theorem \ref{nearS2k1}.
\end{exm}

\section{invariance and near invariance under high order shifts}

In this section, we study the (nearly) invariant subspaces associated with
two higher-order shifts \(S^m\) and \(S^n\), where \(n>m\). We first
characterize the subspaces invariant under both \(S^m,S^n\) and their
adjoints. The essential cases are \(n=mk+\gamma\), where
\(k\geq1\) and \(\gamma\in\{1,\cdots,m-1\}\).    For
\(m\geq3\), every \(f\in H^2(\mathbb D)\) has the decomposition
\[
f(z)=\sum_{l=0}^{m-1}z^l f_l(z^m),
\]
where \(f_l\in H^2(\mathbb D)\). This induces the isometric isomorphism
  $T_m: H^2(\mathbb{D},\mathbb{C}^m)\rightarrow H^2(\mathbb{D})$ by \begin{eqnarray}T_m (f_0(z),f_1(z),\cdots, f_{m-1}(z))=\sum_{l=0}^{m-1}z^lf_{l}(z^m)\label{Tfgm}\end{eqnarray} with $(f_0,\;\cdots,\; f_{m-1})(z) \in H^{2}(\mathbb{D},\mathbb{C}^m).$  Since
 $$ S^m\left( \sum_{l=0}^{m-1}z^lf_{l}(z^{m})\right)=\sum_{l=0}^{m-1}z^l(Sf_{l})(z^{m}).$$  Then the  commutative diagram below holds:
  \begin{eqnarray}\begin{CD}
H^2(\mathbb{D},\mathbb{C}^m)@> S>>H^2(\mathbb{D},\mathbb{C}^m)\\
@VV T_m     V @VV T_m V\\
H^2(\mathbb{D}) @>S^m>> H^2(\mathbb{D}).
\end{CD}\label{commutem}\end{eqnarray}   Using   Theorem \ref{Beurling-Lax}, the $S$-invariant subspace in $H^2(\mathbb{D},\mathbb{C}^m)$
is expressed as $\Theta_{m\times r} H^2(\mathbb{D},\mathbb{C}^r)$ with $\Theta_{m\times r}$  is an operator-valued inner function with values in $\ma{L}(\mathbb{C}^r, \mathbb{C}^m)$ and $0\leq r\leq m$.

\subsection{$S^m$ and $S^{km+\gamma}$-invariant subspaces} In this subsection, we  continue illustrating  the  $S^m$ and $S^{km+\gamma}$-invariant subspaces in $H^2(\mathbb{D})$.

\begin{thm}\label{thm Sm} For $m\geq 3$,  $k\geq 1$ and $\gamma\in \{1,\cdots, m-1\}$, a subspace $\ma{M}$  of $H^2(\mathbb{D})$  is invariant  under both $S^m$ and $S^{km+\gamma}$ if and only if $ \ma{M}=T_m(\Theta_{m\times r} H^2(\mathbb{D},\mathbb{C}^r))$, where  $0\leq r\leq m$, the inner function $\Theta_{m\times r}$ takes values in $\ma{L}(\mathbb{C}^r,
\mathbb{C}^m)$ and satisfies  $\Theta_{m\times r}^{*} \Sigma_{m,\gamma}^k \Theta_{m\times r} \in H^{\infty}_{r \times r}$, and  $T_m$ is an isometric isomorphism as in \eqref{Tfgm}. \end{thm}

\begin{proof}
\emph{Sufficiency.} By Proposition \ref{inv-subspace}, 
$\Theta_{m\times r}H^2(\mathbb D,\mathbb C^r)$ is invariant under the
vector-valued shift. Hence, using \eqref{commutem}, the space
\[
\mathcal M=T_m(\Theta_{m\times r}H^2(\mathbb D,\mathbb C^r))
\]
is invariant under $S^m$.  It behaves as \begin{eqnarray*}
\mathcal M=\left\{(1,\;z,\; \cdots,\;z^{m-1}) \left(\Theta_{m\times r}F\right)(z^m): F \in H^{2}(\mathbb{D},\mathbb{C}^r)\right\}.
 \end{eqnarray*}
Meanwhile, we have that
\begin{eqnarray}
 &&S^{km+\gamma}\mathcal M \nonumber\\&=&\left\{(z^{km+\gamma},\; z^{km+\gamma+1},\; \cdots,\; z^{(k+1)m+\gamma-1}) \left(\Theta_{m\times r}F \right)(z^m):\; F \in H^{2}(\mathbb{D},\mathbb{C}^r)\right\}\nonumber\\&=&
 \left\{( 1, \; z, \; \cdots, \; z^{m-1}) \left(\Sigma_{m,\gamma}^k\Theta_{m\times r}F \right)(z^m): F  \in H^{2}(\mathbb{D},\mathbb{C}^r)\right\}\quad \quad\label{STNm}
 \end{eqnarray}with the inner $m\times m$ matrix  $\Sigma_{m,\gamma}^k$ in \eqref{Sigmam0}. 
Since 
$\Theta_{m\times r}^{*}\Sigma_{m,\gamma}^{k}\Theta_{m\times r}
\in H^\infty_{r\times r}$, the multiplier characterization of
vector-valued Hardy spaces (\cite[Theorem 4.14]{AM02}) implies that
\[
\Sigma_{m,\gamma}^{k}\Theta_{m\times r}
H^2(\mathbb D,\mathbb C^r)
\subset
\Theta_{m\times r}H^2(\mathbb D,\mathbb C^r).
\]
Consequently,
$S^{km+\gamma}\mathcal M\subset\mathcal M$, and hence $\mathcal M$
is invariant under both $S^m$ and $S^{km+\gamma}$.

Conversely, suppose that $\mathcal M$ is   both
$S^m$ and $S^{km+\gamma}$-invariant. By Proposition \ref{inv-subspace}, there
exists an inner matrix function $\Theta_{m\times r}$ such that
\[
\mathcal M=T_m(\Theta_{m\times r}H^2(\mathbb D,\mathbb C^r)).
\]
Using the calculations in \eqref{STNm}, the
$S^{km+\gamma}$-invariance of $\mathcal M$ is equivalent to
\[
\Sigma_{m,\gamma}^{k}\Theta_{m\times r}
H^2(\mathbb D,\mathbb C^r)
\subset
\Theta_{m\times r}H^2(\mathbb D,\mathbb C^r).
\]
Therefore,
$\Theta_{m\times r}^{*}\Sigma_{m,\gamma}^{k}\Theta_{m\times r}$
defines a multiplier on $H^2(\mathbb D,\mathbb C^r)$.
Applying again the multiplier characterization
(\cite[Theorem 4.14]{AM02}), we conclude that
$\Theta_{m\times r}^{*}\Sigma_{m,\gamma}^{k}\Theta_{m\times r}
\in H^\infty_{r\times r}$.
This completes the proof.
\end{proof}

Based on Theorem \ref{thm Sm}, $\ma{M}^{\perp}$ is invariant under $(S^{m})^{*}$ and $(S^{km+\gamma})^*$, i.e.,
 $ \ma{M}^\perp=T_m(K_{\Theta_{m\times r}})$ under the conditions in Theorem \ref{thm Sm}.

\subsection{Near invariance under both $(S^m)^*$ and $(S^{km+\gamma})^*$}  Let $\ma{M}$ be nearly invariant under both $(S^m)^*$ and $(S^{km+\gamma})^*$, with $m\geq3$, $k\geq1$, and $\gamma\in\{1,\cdots,m-1\}$. Define
\[
X_m=\ma{M}\ominus(\ma{M}\cap S^mH^2(\mathbb D)),
\qquad
E_m(z)=(e_0,\cdots,e_{m-1})^{\mathsf t}(z),
\]
where $\{e_j\}_{j=0}^{m-1}$ is obtained by orthonormalizing the reproducing kernel and its first $m-1$ derivative kernels of $\ma{M}$ at $0$. We also set
\[
A(l)=(a_0^l,\cdots,a_{m-1}^l),\qquad l\geq0 .
\]

Firstly, considering the near $(S^m)^*$-invariance,  an arbitrary $f\in \ma{M}$ has the form
$$f(z)= A(0)E_m(z)+S^m f_1(z).$$ Since $f_1\perp e_i$ for $i=0, 1,\cdots, m-1$, the near $(S^m)^*$-invariance implies  $f_1\in \ma{M}$. Meanwhile, it holds that
 $\|f\|^2=|A(0)|^2+\|f_1\|^2.$ So we repeat this process for $f_1$, getting that
\begin{eqnarray}f_1(z)=A(1)E_m(z)+S^m f_2(z)\label{f1m}\end{eqnarray} with $f_2\in \ma{M}$ and
 $\|f_1\|^2=|A(1)|^2+\|f_2\|^2.$  Continuing recursively, we get
$$f_j(z)=A(j)E_m(z)+S^m f_{j+1}(z)$$ with $f_{j+1}\in \ma{M}$ and
 $\|f_j\|^2=|A(j)|^2 +\|f_{j+1}\|^2.$

Combining these equations together, it yields that
$$ f(z)= \left(\sum_{l=0}^{j} (z^{m})^{l}A(l)\right)E_m(z)
+(z^{m})^{(j+1)}f_{j+1}(z)$$ and $$\|f\|^2= \sum_{l=0}^j |A(l)|^2+\|f_{j+1}\|^2.$$
 Utilizing \cite[Lemma 3.3]{BT00} and following the similar proof process of \eqref{S2},  we deduce  that
\begin{align}f(z)&=\left(\sum_{l=0}^{\infty}(z^{m})^{l}A(l) \right)E_m(z)\nonumber\\&=\left(\sum_{l=0}^{\infty} a_{0}^l(z^{m})^l, \;\sum_{l=0}^{\infty}a_1^l(z^{m})^l, \; \cdots,\; \sum_{l=0}^{\infty} a_{m-1}^l(z^{m})^l\right)E_m(z),\;z\in\mathbb{D},\label{S2m}
\end{align} where the sums converge in $H^2$ norm and  $\|f\|^2=
 \sum_{l=0}^{\infty}|A(l)|^2.$
Observing from  \eqref{f1m} and \eqref{S2m}, it holds that \begin{eqnarray}f_1(z)= \left(\sum_{l=0}^{\infty}
  (z^{m})^{l}A(l+1)\right)E_m(z) \in \ma{M}.\label{f11m}\end{eqnarray}
Hereafter, we denote \begin{eqnarray} \Phi(z):= (\phi_0, \;\phi_1,\;\cdots,\;\phi_{m-1})(z)\label{Phi0m}\end{eqnarray}with
$\phi_i(z)=\sum_{l=0}^\infty a_i^lz^{l}$ for $0\leq i\leq m-1.$ By \eqref{S2m}, we  alternatively express that $f\in \ma{M}$ if and only if  $f(z)=\Phi(z^m)E_m(z)$ with  \begin{eqnarray}\|f\|^2= \|\Phi\|^2, \label{normm}\end{eqnarray}
and $\Phi(z)$ lies in an $S^*$-invariant subspace $K\subset H^2(\mathbb{D},\mathbb{C}^m)$. The closedness of $K$ can be deduced from \eqref{normm}, and  \eqref{f11m} entails $K$ is $S^*$-invariant.

Conversely, if $$\ma{M}=\left\{\Phi(z^m)E_m(z):\; \Phi(z)\in K\right\}$$ is a closed subspace of $H^2(\mathbb{D}),$ where $K$  is  $S^*$-invariant in $H^2(\mathbb{D},\mathbb{C}^m)$, then $\ma{M}$ is nearly $(S^m)^*$-invariant in $H^2(\mathbb{D})$.

Since $K$ is $S^*$-invariant in $H^2(\mathbb{D},\mathbb{C}^m)$,  Theorem \ref{Beurling-Lax} implies
$$K=K_{\Theta_{m\times r}}:=H^2(\mathbb{D},\mathbb{C}^m)\ominus \Theta_{m\times r}H^2(\mathbb{D},\mathbb{C}^r)$$ with    $0\leq r\leq m$ and  the inner function $\Theta_{m\times r}$ takes values in $\ma{L}(\mathbb{C}^r,\mathbb{C}^m)$. 
So there exists a unitary map  $J_m:\;\mathcal{M}\rightarrow K_{\Theta_{m\times r}}$ defined by
$$J_m\left(G(z^m) E_m(z)\right)=G(z),\;\mbox{for}\; \;G\in H^2(\mathbb{D},\mathbb{C}^m).$$

In sum, we characterize the nearly $(S^m)^*$-invariant subspaces in $H^2(\mathbb{D})$ as below.
\begin{prop}\label{prop nearSm}For $m\geq 3$, a subspace $\ma{M}$ of $H^2(\mathbb{D})$ is nearly $(S^m)^*$-invariant if and only if  there exists a uitary map
$$J_m:\;\mathcal{M}\rightarrow K_{\Theta_{m\times r}}\;\; \mbox{defined by}\;G(z^m) E_m(z)\mapsto G(z),$$ where  the inner function $\Theta_{m\times r}$ takes values in $  \ma{L}(\mathbb{C}^r,\mathbb{C}^m)$  for $0\leq r\leq m$.
\end{prop}

Next we present the nearly $(S^m)^*$ and $(S^{km+\gamma})^*$-invariant subspaces, which also include all nearly invariant subspaces for the non-cyclic shift semigroup $\{((S^m)^*)^n ((S^{km+\gamma})^*)^l:\; n,\;l\in \mathbb{N}_0,\;k\geq 1\}$.

\begin{thm}\label{nearSmr}For $m\geq 3$,  $k\geq 1$ and $\gamma\in \{1,\cdots, m-1\}$, a subspace $\ma{M}$ of $H^2(\mathbb{D})$ is  nearly $(S^m)^*$  and $(S^{km+\gamma})^*$-invariant  if and only if there exists a unitary map
\begin{align*}J_m:\;\mathcal{M}\rightarrow K_{\Theta_{m\times r}}\;\; \mbox{defined by}\;G(z^m) E_m(z)\mapsto G(z),\end{align*} where the inner function $\Theta_{m\times r}$ takes values in $\ma{L}(\mathbb{C}^r,\mathbb{C}^m)$  and satisfies  $\Theta_{m\times r}^{*} \Sigma_{m,\gamma}^k \Theta_{m\times r} \in H^{\infty}_{r \times r}$  for $0\leq r\leq m$.
 \end{thm}

\begin{proof}
By Proposition \ref{prop nearSm}, the unitary map
$J_m:\mathcal M\to K_{\Theta_{m\times r}}$ gives
$f(z)=\Phi(z^m)E_m(z)\in \ma{M}$ with
$\Phi=(\phi_0,\cdots,\phi_{m-1})\in K_{\Theta_{m\times r}}$ as in \eqref{Phi0m}.

Assume that
$\Theta_{m\times r}^{*}\Sigma_{m,\gamma}^{k}
\Theta_{m\times r}\in H^\infty_{r\times r}$ and that
$f\in\mathcal M$ satisfies
$f^{(n)}(0)=0$ for $0\leq n\leq km+\gamma-1$.
By the definition of $\Sigma_{m,\gamma}^k$ and the relation
\eqref{commutem}, the condition
$(S^{km+\gamma})^*f\in\mathcal M$ is equivalent, under $J_m$, to
\begin{align}
(\Sigma_{m,\gamma}^{k})^*\Phi^{\mathsf t}
\in K_{\Theta_{m\times r}} .\label{aimmk}
\end{align}

We verify this inclusion. For any
$F\in H^2(\mathbb D,\mathbb C^r)$, the assumption on
$\Theta_{m\times r}$ implies that
$\Sigma_{m,\gamma}^{k}\Theta_{m\times r}F
=\Theta_{m\times r}G$ for some
$G\in H^2(\mathbb D,\mathbb C^r)$. Hence,
\[
\begin{aligned}
\left\langle
(\Sigma_{m,\gamma}^{k})^*\Phi^{\mathsf t},
\Theta_{m\times r}F
\right\rangle
&=
\left\langle
\Phi^{\mathsf t},
\Sigma_{m,\gamma}^{k}\Theta_{m\times r}F
\right\rangle  \\
&=
\left\langle
\Phi^{\mathsf t},
\Theta_{m\times r}G
\right\rangle=0 ,
\end{aligned}
\]
because $\Phi\in K_{\Theta_{m\times r}}$. Therefore, \eqref{aimmk} holds
and hence $\mathcal M$ is nearly
$(S^{km+\gamma})^*$-invariant.

Conversely, if $\mathcal M$ is nearly
$(S^{km+\gamma})^*$-invariant, then
$K_{\Theta_{m\times r}}$ is invariant under
$(\Sigma_{m,\gamma}^{k})^*$. Equivalently,
$\Sigma_{m,\gamma}^{k}\Theta_{m\times r}
H^2(\mathbb D,\mathbb C^r)
\subset
\Theta_{m\times r}H^2(\mathbb D,\mathbb C^r)$.
By the multiplier characterization of vector-valued Hardy spaces
(\cite[Theorem 4.14]{AM02}), this is equivalent to
$\Theta_{m\times r}^{*}\Sigma_{m,\gamma}^{k}
\Theta_{m\times r}\in H^\infty_{r\times r}$.
This completes the proof.
\end{proof}

As an application, we obtain simultaneous (near) invariance results for
discrete semigroups generated by several elements. More general cases
can be treated similarly.

\begin{rem}
Let $m\geq2$ and $n=k_nm+\gamma_n<l=k_lm+\gamma_l$, where
$n\nmid l$, $k_n,k_l\geq1$, and
$\gamma_n,\gamma_l\in\{1,\cdots,m-1\}$. The simultaneous invariance
under $S^m,S^n,S^l$, as well as the  near invariance under
$(S^m)^*,(S^n)^*,(S^l)^*$, follows directly from Theorem \ref{thm Sm}
and Theorem \ref{nearSmr}, respectively, by imposing
$\Theta_{m\times r}^{*}\Sigma_{m,\gamma_n}^{k_n}\Theta_{m\times r},
\Theta_{m\times r}^{*}\Sigma_{m,\gamma_l}^{k_l}\Theta_{m\times r}
\in H^\infty_{r\times r}$.
\end{rem}

This section ends with  a subspace simultaneously
invariant under $S^3,S^4$ and $S^5$.

\begin{exm}
Let $\mathcal M=\operatorname{span}\{1,z^3,z^4,z^5,\cdots\}
=\mathbb C\oplus z^3H^2(\mathbb D)$. Clearly, $\mathcal M$ is neither
$S$  nor $S^2$-invariant. Using the decomposition
$f(z)=f_0(z^3)+zf_1(z^3)+z^2f_2(z^3)$, every $g\in\mathcal M$ admits a
representation
 \begin{align*}g(z)&=g(0)+z^3(f_0(z^3)+zf_1(z^3)+z^2f_2(z^3))\\&=
g(0)+(Sf_0)(z^3)+z(Sf_1)(z^3)+z^2(Sf_2)(z^3)
\\&:=g_0(z^3)+zg_1(z^3)+z^2g_2(z^3),\end{align*}
where  $g_0(z)=g(0)+(zf_0) (z),\;g_1(z)= (zf_1) (z),\; g_2(z)=(zf_2)(z).$  Therefore,
$\mathcal M=T_3(\Theta_{3\times3}H^2(\mathbb D,\mathbb C^3))$, where
\[
\Theta_{3\times3}
=
\begin{pmatrix}
1&0&0\\
0&z&0\\
0&0&z
\end{pmatrix}.
\]

A direct computation gives  $$\Theta_{3\times 3}^* \Sigma_{3,1}^1\Theta_{3\times 3}=\left(\begin{array}{ccc}0& 0 &z^3\\1& 0& 0\\0 &z^2 &0\end{array}\right), \;\Theta_{3\times 3}^* \Sigma_{3,2}^1\Theta_{3\times 3}=\left(\begin{array}{ccc}0& z^3 &0\\ 0& 0& z^2\\1 &0 &0\end{array}\right)$$
and both matrices belong to $H^\infty_{3\times3}$. Hence, by
Theorem \ref{thm Sm}, $\mathcal M$ is invariant under
$S^3,S^4$ and $S^5$.  Then $\mathcal M^\perp=\operatorname{span}\{z,z^2\}
=T_3(K_{\Theta_{3\times3}})$ is invariant under
$(S^3)^*,(S^4)^*$ and $(S^5)^*$.
\end{exm}

\section{(Near) invariance for Toeplitz operators induced by Blaschke products}

Let $B_m$ be an $m$-degree Blaschke product,
\[
B_m(z)=\lambda\prod_{j=1}^{m}
\frac{z-z_j}{1-\overline{z_j}z},
\qquad |\lambda|=1,\quad |z_j|<1 .
\]
The model space $K_{B_m}=H^2(\mathbb D)\ominus B_mH^2(\mathbb D)$ has
dimension $m$. By the Wold decomposition,
\[
H^2(\mathbb D)=\bigoplus_{j\ge0}B_m^jK_{B_m},
\qquad
H^2(\mathbb D,\mathbb C^m)
=\bigoplus_{j\ge0}z^jK,
\]
where $K=H^2(\mathbb D,\mathbb C^m)\ominus zH^2(\mathbb D,\mathbb C^m)$.

Let $\{e_j\}_{j=1}^{m}$ and $\{\delta_j^m\}_{j=1}^{m}$ be orthonormal
bases of $K_{B_m}$ and $K$, respectively. The corresponding unitary
operator
$U:H^2(\mathbb D,\mathbb C^m)\rightarrow H^2(\mathbb D)$ is determined by
\[
U(B_m^je_l)=z^j\delta_l^m,\qquad j\ge0,\ 1\le l\le m .
\]
Together with the isometric isomorphism $T_m$ in \eqref{Tfgm}, we have
the intertwining relation
\[
S^m(T_mU)=(T_mU)T_{B_m}.
\]
Hence, for every $n\ge1$,
\[
(S^m)^n=(T_mU)(T_{B_m})^n(U^*T_m^{-1}).
\]

This unitary equivalence immediately gives the following reduction.

\begin{prop}\label{Bmn}
Let $m\ge2$ and $n\ge1$, a subspace $\mathcal M$ of $H^2(\mathbb D)$
is $(S^m)^n$-invariant (resp. nearly $(S^m)^{*n}$-invariant)
if and only if $(U^*T_m^{-1})\mathcal M$ is
$(T_{B_m})^n$-invariant (resp.  nearly
$(T_{B_m})^{*n}$-invariant).
\end{prop}

We next consider the simultaneous invariance under two powers of
$T_{B_m}$. Let $n<l$ and write
$l=nk+\gamma$, where $k\ge1$ and
$\gamma\in\{1,\cdots,n-1\}$. By Proposition \ref{Bmn}, this problem is
equivalent to the simultaneous invariance under
$S^{mn}$ and $S^{mnk+m\gamma}$. Therefore Theorem \ref{thm Sm}  with $m$ replaced by $mn$ and $\gamma$ replaced by
$m\gamma$ yield the following result.

\begin{cor}
Let $m\ge1$, $n\ge2$, $k\ge1$, and
$\gamma\in\{1,\cdots,n-1\}$, a subspace $\mathcal M$ of $H^2(\mathbb{D})$ is invariant under
both $(T_{B_m})^n$ and $(T_{B_m})^{kn+\gamma}$ if and only if
\[
\mathcal M=(U^*T_m^{-1})
\big[T_{mn}(\Theta_{mn\times r}H^2(\mathbb D,\mathbb C^r))\big],
\]
where $0\le r\le mn$, the inner function $\Theta_{mn\times r}$ takes values in $\mathcal{L}(\mathbb{C}^r, \mathbb{C}^{mn})$ and satisfies
$\Theta_{mn\times r}^{*}
\Sigma_{mn,m\gamma}^{k}
\Theta_{mn\times r}
\in H^\infty_{r\times r}. $
Moreover,
$ \mathcal M^\perp=(U^*T_m^{-1})
[T_{mn}(K_{\Theta_{mn\times r}})]
$  is invariant under both $(T_{B_m}^*)^n$ and
$(T_{B_m}^*)^{kn+\gamma}$.
\end{cor}

Similarly, Proposition \ref{Bmn} together with Theorem \ref{nearSmr}
give  the corresponding characterization of nearly invariant subspaces.

\begin{cor}
Let $m\ge1$, $n\ge2$, $k\ge1$, and
$\gamma\in\{1,\cdots,n-1\}$, a subspace $\mathcal M$ of $H^2(\mathbb{D})$ is nearly invariant under
both $(T_{B_m}^*)^n$ and $(T_{B_m}^*)^{kn+\gamma}$ if and only if there exists a unitary map
\[
J_{mn}:\;(T_mU)\mathcal M\longrightarrow K_{\Theta_{mn\times r}},
\qquad
G(z^{mn})E_{mn}(z)\mapsto G(z),
\]
where $0\le r\le mn$, $\Theta_{mn\times r}$ is inner  and satisfies $
\Theta_{mn\times r}^{*}
\Sigma_{mn,m\gamma}^{k}
\Theta_{mn\times r}
\in H^\infty_{r\times r}.$  Here $E_{mn}$ is obtained by orthonormalizing the reproducing kernel and
the first $mn-1$ derivative kernels at the origin.
\end{cor}

\textbf{Acknowledgements}\;\; The work was supported in part by the National Natural Science Foundation of China (Grant No. 12471126).


\begin{thebibliography}{}

\bibitem{AM02} J. Agler and J.E. McCarthy, Pick Interpolation and Hilbert Function Spaces, Graduate Studies in Mathematics, 44, American Mathematical Society, Providence, RI, 2002.

\bibitem{ABBH}A. Aleman,  A. Baranov, Y. Belov and H.  Hedenmalm, Backward shift and nearly invariant subspaces of Fock-type spaces, Int. Math. Res. Not. IMRN, 10(2022), 7390-7419.

\bibitem{ALS}K. Arshad, S. Lata  and D. Singh, Nearly invariant Brangesian subspaces,  Canad. Math. Bull., 68(1)(2025), 318--337.

\bibitem{BT00} C. Benhida and D.Timotin,  Finite rank perturbations of contractions, Integral Equations and Operator Theory, 36(3) (2000), 253--268.  

\bibitem{CaP1} M.C. C\^{a}mara and J.R. Partington, Near invariance and kernels of Toeplitz operators, Journal d'Analyse Math., 124(2014), 235--260.

\bibitem{CaP} M.C. C\^{a}mara and J.R. Partington, Finite-dimensional Toeplitz kernels and nearly-invariant subspaces, J. Operator Theory, 75(1)(2016), 75--90.

\bibitem{CaP2}  M.C. C\^{a}mara and J.R. Partington, Toeplitz kernels and model spaces. The diversity and beauty of applied operator theory, 139-153, Oper. Theory Adv. Appl., 268, Birkh\"auser/Springer, Cham, 2018.



\bibitem{CCP10} I. Chalendar, N. Chevrot and J.R. Partington,   Nearly invariant subspaces for backwards shifts on vector-valued Hardy spaces, J. Operator Theory, 63(2)(2010), 403--415.











\bibitem{GGP22}
E.A. Gallardo-Guti\'errez and J.R. Partington, Multiplication by a finite Blaschke product on weighted Bergman spaces: commutant and reducing subspaces, J. Math. Anal. Appl. 515(1)(2022),   Paper No. 126383.


\bibitem{GMR} S. Garcia, J. Mashreghi and W.T. Ross, Introduction to Model Spaces and Their Operators, Cambridge: Cambridge University Press, 2016.

\bibitem{HR}A. Hartmann and W.T. Ross, Truncated Toeplitz operators and boundary values in nearly invariant subspaces, Complex Anal. Oper. Theory, 7(1)(2013), 261--273.

\bibitem{Ha} E. Hayashi, The kernel of a  Toeplitz operator, Integral Equations and Operator Theory, 9(4)(1986), 588--591.


\bibitem{hitt} D. Hitt, Invariant subspaces of ${\ma H}^2$ of an annulus, Pacific J. Math., 134(1)(1988), 101--120.

\bibitem{LS97}
T.L. Lance and M.I. Stessin,   Multiplication invariant subspaces of Hardy spaces, Canad. J. Math., 49(1)(1997), 100--118.

\bibitem{Lax}
P.D. Lax,  Translation invariant spaces, Acta Math., 101(1959), 163--178.

\bibitem{LP2} Y. Liang and J.R. Partington, Nearly invariant subspaces for operators in Hilbert spaces, Complex Anal. Oper. Theory, 15(5)(2021), Paper No. 5.

\bibitem{LP3} Y. Liang and J.R. Partington, Nearly invariant subspaces for shift semigroups,  Sci. China Math., 65(9)(2022), 1895--1908.

\bibitem{LP4} Y. Liang and J.R. Partington, Cyclic nearly invariant
subspaces for semigroups of isometries, Math. Z., 307 (2024), Paper No. 58.

\bibitem{LP5} Y. Liang and J.R. Partington, Composition
operators between Toeplitz kernels,  Collect. Math. (2025), https://doi.org/10.1007/s13348-025-00495-7.

\bibitem{Nik} N. Nikolski, Operators,  Functions, and Systems: an easy reading. Vol. 1, volume 92 of Mathematical Surveys and Monographs, American Mathematical Society, Providence, RI, 2002. Hardy, Hankel, and Toeplitz, Translated from the French by Andreas Hartmann.

\bibitem{ryan}
R. O'Loughlin, Nearly invariant subspaces with applications to truncated Toeplitz operators, Complex Anal. Oper. Theory, 14(8)(2020), 10--24.

\bibitem{Par} J.R. Partington,  Linear Operators and Linear Systems, London Mathematical Society Student Texts, 60. Cambridge University Press, Cambridge, 2004.






\bibitem{Sa1} D. Sarason, Nearly invariant subspaces of the backward shift, Contributions to operator theory and its applications (Mesa, AZ, 1987), 481--493, Oper. Theory Adv. Appl. 35, Birkh\"{a}user, Basel, 1988.

\bibitem{Sa2} D. Sarason, Kernels of Toeplitz operators, Oper. Theory: Adv. Appl., 71(1994), 153--164.

\end{thebibliography}
\end{document}